\documentclass[12pt]{amsart}

\usepackage{geometry}
\usepackage{amssymb}
\usepackage{amsmath}
\usepackage{amsfonts}
\usepackage[utf8]{inputenc}
\usepackage{amsthm}
\usepackage{tikz-cd}
\usepackage{float}
\usepackage{diagbox}

\usepackage{hyperref}

\newtheorem{thm}{Theorem}[section]
\newtheorem{lemma}[thm]{Lemma}

\newtheorem{prop}[thm]{Proposition}

\newtheorem{cor}[thm]{Corollary}

\newcommand{\Z}{\mathbb{Z}}

\newcommand{\Q}{\mathbb{Q}}
\newcommand{\N}{\mathbb{N}}

\newcommand{\cal}{\mathcal}

\newcommand{\ve}{\varepsilon}

\title[On continued fraction partial quotients of square roots of primes]{On continued fraction partial quotients \\ of square roots of primes} 

\author{V\' \i t\v ezslav Kala}
\address{Charles University, Faculty of Mathematics and Physics, Department of Algebra, Sokolov\-sk\' a 83, 18600 Praha~8, Czech Republic}
\email{vitezslav.kala@matfyz.cuni.cz}

\author{Piotr Miska}
\address{Jagiellonian University in Cracow, Faculty of Mathematics and Computer Science, Institute of
Mathematics, {\L}ojasiewicza 6, 30-348 Krak\'ow, Poland}
\email{piotr.miska@uj.edu.pl}

\thanks{
	V.K. was supported by Czech Science Foundation (GA\v CR) grant 21-00420M and Charles University Research Centre program UNCE/SCI/022. 
	P.M. was supported by the grant of the Polish National Science Centre no. UMO-2019/34/E/ST1/00094.}

\begin{document}

\begin{abstract}
    We show that for each positive integer $a$ there exist only finitely many prime numbers $p$ such that $a$ appears an odd number of times in the period of continued fraction of $\sqrt{p}$ or $\sqrt{2p}$. We also prove that if $p$ is a prime number and $D=p$ or $2p$ is such that the length of the period of continued fraction expansion of $\sqrt{D}$ is divisible by $4$, then $1$ appears as a partial quotient in the continued fraction of $\sqrt{D}$. Furthermore, we give an upper bound for the period length of continued fraction expansion of $\sqrt{D}$, where $D$ is a positive non-square, and factorize some family of polynomials with integral coefficients connected with continued fractions of square roots of positive integers. These results answer several questions recently posed by Miska and Ulas \cite{MiUl}.
\end{abstract}

\maketitle

\section{Introduction}

Each irrational real number $x$ can be written in the form of infinite continued fraction
$$
x=a_{0}+\cfrac{1}{a_{1}+\cfrac{1}{a_{2}+\cfrac{1}{a_{3}+\cfrac{1}{\ddots}}}}=[a_{0},a_{1},a_{2},\ldots],
$$
where the partial quotients $a_{0},a_{1},\ldots $ can be recursively computed in the following way
$$
\alpha_0=x,\;a_{0}=\lfloor \alpha_{0}\rfloor,\quad \alpha_{k+1}=\frac{1}{\alpha_{k}-a_{k}},\quad a_{k}=\lfloor \alpha_{k}\rfloor.
$$
It is clear that $a_{0}\in\Z$ and $a_{i}\in\N_{+}$ for $i\in\N_{+}$ (we denote by $\Z$ the set of all integers, and by $\N_0$ and $\N_{+}$ the subsets of non-negative and positive integers, respectively). The number $a_{i}$ is called the $i$-th partial quotient (or $i$-th coefficient) of the continued fraction for $x$.

The continued fraction $[a_{0},a_{1},a_{2},\ldots]$ is eventually periodic if there exists $l\in\N_{+}$ such that for all sufficiently large $n$ we have $a_{n}=a_{n+l}$; the smallest such $l$ is the period of the continued fraction. Then we write $[a_{0},a_{1},a_{2},\ldots,a_{n-1},\overline{a_{n},\ldots,a_{n-1+l}}]$. A famous result from the theory of continued fractions, Lagrange's theorem, states that the irrational number $x$ has eventually periodic continued fraction if and only if $x$ is a quadratic irrational. In particular, for any given non-square $D\in\N_{+}$ the continued fraction of $\sqrt{D}$ is eventually periodic. Moreover, it is well known that $\sqrt{D}=[a_0,\overline{a_1,\ldots,a_l}]$, where $a_0=\lfloor\sqrt{D}\rfloor$, $a_l=2\lfloor\sqrt{D}\rfloor$ and $a_{l-j}=a_j$ for each $j\in\{1,\ldots,l-1\}$. Denote the length $l$ of the period of the continued fraction of $\sqrt{D}$ by $T_D$. 

\medskip

There has been great interest in studying the behaviour of partial quotients of continued fractions. Nevertheless, even in the case of $x$ being any particular non-quadratic algebraic irrationality (even $x=\sqrt[3]{2}$) we do not have information about unboundedness of the partial quotients, nor about infinite appearance of a given positive integer as a partial quotient of continued fraction of $x$. 

Although the problem of description of partial quotients of continued fractions of quadratic irrationalities is well studied (see e.g. \cite[Drittes Kapitel]{Pe}), it is still far from completely understood, as evidenced by recent results \cite{CS, DCS, Lou20, Lou} on periods and partial quotients of square roots of prime and semi-prime numbers. Also, the papers \cite{Fr, Gol, HK, RipTay} are devoted to continued fractions of square roots of positive integers with a given period length, or even all the period given without the last partial quotient $a_l$ (which equals $2a_0$). Continued fractions also motivated some recent results on the distribution of class numbers of real quadratic fields \cite{CF+}.

A probabilistic approach shows that $1$ is the most probable value of partial quotient of continued fraction of a uniformly randomly chosen irrational number $x=[a_0,a_1,a_2,\ldots]$. Gauss conjectured and Kuzmin \cite{Kuz} proved that if $k\to\infty$ then the probability that $a_k=n$ tends to $\log_2\frac{(n+1)^2}{n(n+2)}$. Moreover,  $a_1=1$ with probability $1/2$. 

Inspired by a result of Skałba \cite{Ska}, Miska and Ulas \cite{MiUl} showed by various approaches that there exist infinitely many prime numbers $p$ with a prescribed sequence $(a_1,\ldots, a_k)$ of beginning partial quotient of continued fractions of their square roots. The aim of our paper is to answer some questions posed in \cite{MiUl}.

\medskip

Define for each $i\in\N$ the set $\cal{L}_i$ as the set of prime numbers $p$ for which $1$ appears exactly $i$ times in the period of continued fraction of $\sqrt{p}$. According to numerical computations the authors of \cite{MiUl} asked (\cite[Question 5.4]{MiUl}) whether $\cal{L}_1=\{3\}$, $\cal{L}_3=\{7\}$ and $\cal{L}_{2i+1}=\varnothing$ for each integer $i\geq 2$. 
We answer this question in the affirmative as the special case $a=1$ of the following theorem.

\begin{thm}\label{TA}
Let $a\in\N_+$ and let $p$ a prime number. Then the period of continued fraction of $\sqrt{p}$ contains $a$ as a partial quotient an odd number of times if and only if
\begin{enumerate}
\item $p\equiv 3\pmod{4}$ and $a^2<p<(a+2)^2$ when $a$ is odd; 
\item $\frac{a^2}{4}<p<\frac{(a+2)^2}{4}$ when $a$ is even.
\end{enumerate}
The period of continued fraction of $\sqrt{2p}$ contains $a$ as a partial quotient an odd number of times if and only if $a$ is even and $p\in\left(\frac{a^2}{8},\frac{(a+2)^2}{8}\right)$ or $p\in\left(\frac{a^2}{2},\frac{(a+2)^2}{2}\right)$ and $2\mid T_{2p}$.
\end{thm}

Numerical search in \cite{MiUl} also revealed another phenomenon. Namely, there is no prime $p$ in the set $\cal{L}_0\cap [1,10^7] $ with $T_p$ divisible by $4$. Accordingly, \cite[Question 5.3]{MiUl} asked about the existence of a prime $p$ with $T_p$ divisible by $4$ and without $1$ as a continued fraction partial quotient of $\sqrt{p}$. We shall prove that such a $p$ does not exist.

\begin{thm}\label{TB}
Let $p$ a prime number and $D\in\{p,2p\}$. Assume that $T_D$ is divisible by $4$. Then $1$ appears as a partial quotient in the periodic part of continued fraction of $\sqrt{D}$.
\end{thm}

Another question \cite[Question 5.1]{MiUl} based on numerical experiments concerned an upper bound for $T_p$. Specifically, denoting by $p_m$ the $m$th prime number,  it asked for the value of $\limsup_{m\to\infty}\frac{T_{p_m}}{\sqrt{m}\log m}$ and if this quotient is less than $1$ for any $m\geq 5$. We partially answer this question thanks to the following theorem.

\begin{thm}\label{th:period length}
Let $D>1$ a squarefree integer, $\sqrt{D}=[a_0,\overline{a_1,\ldots,a_l}]$, and $\chi$ the corresponding quadratic character.
Then $$T_D=l<\frac 4{\log 2}\sqrt D \cdot L(1,\chi)\ll
\begin{cases}
  \sqrt D \log D  &  \text{ unconditionally}, \\
  \sqrt D \log\log D  &  \text{ assuming GRH}.
\end{cases}$$
\end{thm}

Such a theorem is not new; similar results were obtained, e.g., \cite{Pod,Poh,Sa} (even with somewhat better constants than ours). However, we include the proof for the sake of completeness, and then discuss what this implies for \cite[Question 5.1]{MiUl}.

Let us briefly describe the content of the paper. Section \ref{sec:2} contains some basic and auxiliary results that will be used in the sequel. Sections \ref{sec:3}, \ref{sec:4}, and \ref{sec:5} give the proofs of the results announced above with some additional remarks and consequences; in particular, at the end of Section \ref{sec:5} we specialize Theorem \ref{th:period length} to the case $D=p_m$ and discuss the answer to \cite[Question 5.1]{MiUl}. Section \ref{factorpoly} is devoted to the factorization of some quadratic polynomials naturally connected with continued fractions of square roots of positive integers. Our result is complementary to those ones obtained in \cite{Fr, MiUl, RipTay}.

Some of the proof techniques and lemmas that we use are known (and appear, e.g., in Perron's classical book \cite{Pe}). Nevertheless, we often give the proofs of such results for convenience of the readers and for the sake of completeness of our paper.

\section{Preliminaries}\label{sec:2}

Let us start by recalling the following result that is a part of Theorem 1 in \cite{Lou}.

\begin{thm}[{\cite[Theorem 1]{Lou}}]\label{T0}
Let $p\equiv 3\pmod{4}$ a prime number and $D\in\{p,2p\}$. If $T_D=l$ and $\sqrt{D}=[a_0,\overline{a_1,\ldots,a_l}]$, then $T_D=2L$ for some $L\in\N_+$ and $a_L$ is an integer from the set $\{\lfloor\sqrt{p}\rfloor-1,\lfloor\sqrt{p}\rfloor\}$ with the same parity as $D$. Moreover, $L$ is even if and only if $p\equiv 7\pmod{8}$.
\end{thm}

For the next result we need some preparation. Let $D\in\N_+$ not a perfect square and $T_D=l$. Write $\omega_0=\sqrt{D}=[a_0,\overline{a_1,\ldots,a_l}]$ and $\omega_k=[\overline{a_k,\ldots,a_l,a_1,\ldots,a_{k-1}}]$ for $k\in\N_+$. One can easily show by induction on $k$ that $\omega_k=\frac{\sqrt{D}+P_k}{Q_k}$ for some $P_k,Q_k\in\N_0$ such that $P_k^2<D$ and $Q_k\mid D-P_k^2$. If additionally $k>0$, then $P_k,Q_k\in\N_+$. Moreover, we have recurrence relations
\begin{equation}\label{1}
    P_{k+1}=a_kQ_k-P_k,
\end{equation}
\begin{equation}\label{2}
    Q_{k+1}=\frac{D-P_{k+1}^2}{Q_k}=\frac{D-P_k^2}{Q_k}+2a_kP_k-a_k^2Q_k
\end{equation}
(for these and other properties, see \cite[\S 23--25]{Pe}).

\begin{lemma}\label{L0}
    Let $D\in\N_+$ not a perfect square. With the above notation, if $T_D=2L$ for some $L\in\N_+$, then 
    \begin{equation}\label{2.1}
    P_{L+1}=P_L
    \end{equation}
    and
    \begin{equation}\label{2.2}
    4Q_LQ_{L+1}=4D-a_L^2Q_L^2.
    \end{equation}
\end{lemma}

\begin{proof}
Consider
$$\omega_L=[\overline{a_L,a_{L-1}\ldots,a_1,a_l,a_1,\ldots,a_{L-1}}]$$
and
$$\omega_{L+1}=[\overline{a_{L-1}\ldots,a_1,a_l,a_1,\ldots,a_{L-1},a_L}].$$
Since the period of continued fraction of $\omega_{L+1}$ is the reverse of the period of continued fraction of $\omega_L$, we know from Galois theorem that $\omega_L'\omega_{L+1}=-1$, where $\omega_L'=\frac{-\sqrt{D}+P_L}{Q_L}$ denotes the algebraic conjugate of $\omega_L$. Thus, we have 
$$\frac{P_L-\sqrt{D}}{Q_L}\frac{P_{L+1}+\sqrt{D}}{Q_{L+1}}=-1,$$ 
equivalently 
$$(D-P_LP_{L+1})+(P_{L+1}-P_L)\sqrt{D}=Q_LQ_{L+1}.$$ 
This means that $P_{L+1}=P_L$. Hence, from \eqref{1} we have $2P_L=a_LQ_L$. Applying this to \eqref{2} with $k=L$, we get \eqref{2.2}.
\end{proof}

Let us now state the following result that is a part of the proof of \cite[Satz 21, page 108]{Pe}. However, we include its proof for readers' convenience.

\begin{prop}\label{T1}
Let $p$ a prime number and $D\in\{p,2p\}$. With the above notation, if $T_D=2L$ for some $L\in\N_+$, then $Q_L=2$.
\end{prop}

\begin{proof}
Note that $T_2=1$ and $2\cdot 2$ is a perfect square, and so we may assume that $p$ is an odd prime number. 

By Lemma \ref{L0} we have $Q_L\mid 4D$. Let us first observe that $4\nmid Q_L$. Indeed, if $4\mid Q_L$, then \eqref{2.2} implies that $4\mid D$ -- a contradiction. Thus $Q_L\mid D\mid 2p$. 

On the other hand, $Q_L\mid D-P_L^2\in\{1,\ldots,D-1\}$. As result, $Q_L\in\{1,2,p\}$. However, $Q_L=1$ is not possible, as in that case, $\omega_L=\sqrt{D}+P_L$, so $\omega_{k+L}=\omega_k$ for $k\in\N_+$. Then $L$ would be a multiple of the length of period $T_D=2L$, which is a contradiction. If $Q_L=p$, then $D=2p$, but in this case we have $p\mid 2p-P_L^2$, which means that $p\mid P_L$. This is impossible as $0<P_L<\sqrt{2p}<p$. Finally we conclude that $Q_L=2$.
\end{proof}

Another consequence of Lemma \ref{L0} is the following.

\begin{prop}\label{P1}
Let $D\in\N_+$ not a perfect square. If $T_D=2L$ for some $L\in\N_+$, $2\mid Q_L$, and $4\nmid Q_L$, then $a_L\equiv D\pmod{2}$.
\end{prop}

\begin{proof}
After dividing \eqref{2.2} by $4$ we get
$$Q_LQ_{L+1}=D-a_L^2\cdot\left(\frac{Q_L}{2}\right)^2.$$
Since $Q_L$ is even and $\frac{Q_L}{2}$ is odd, we have $a_L\equiv D\pmod{2}$.
\end{proof}

The condition $Q_L=2$, being the assertion of Proposition \ref{T1}, has immediate but serious consequence.

\begin{prop}\label{P2}
Let $D\in\N_+$ not a perfect square. If $T_D=2L$ for some $L\in\N_+$ and $Q_L=2$, then $P_L=a_L\in\{\lfloor\sqrt{D}\rfloor -1, \lfloor\sqrt{D}\rfloor\}$.
\end{prop}

\begin{proof}
By \eqref{1} and Lemma \ref{L0} we have
$$P_L=2a_L-P_L,$$
so $P_L=a_L$. Recall that
$$a_L=\lfloor\omega_L\rfloor =\left\lfloor\frac{\sqrt{D}+P_L}{Q_L}\right\rfloor =\left\lfloor\frac{\sqrt{D}+a_L}{2}\right\rfloor .$$
Thus 
$$a_L\leq\frac{\sqrt{D}+a_L}{2}<a_L+1.$$
Equivalently,
$$2a_L\leq\sqrt{D}+a_L<2a_L+2$$
or
$$\sqrt{D} -2<a_L\leq\sqrt{D},$$
which was to prove.
\end{proof}

As a conclusion from Propositions \ref{T1}, \ref{P1}, and \ref{P2} we obtain the following generalization of Theorem \ref{T0}.

\begin{thm}\label{T2}
Let $p$ a prime number and $D\in\{p,2p\}$. If $T_D=2L$ for some $L\in\N_+$ and $\sqrt{D}=[a_0,\overline{a_1,\ldots,a_l}]$, then $Q_L=2$ and $P_L=a_L$ is an integer from the set $\{\lfloor\sqrt{D}\rfloor-1,\lfloor\sqrt{D}\rfloor\}$ with the same parity as $D$.
\end{thm}

Now we can also quite quickly reprove the following well-known fact (see, e.g. \cite[Satz 22, page 108]{Pe}).

\begin{cor}\label{C1}
If $p\equiv 1\pmod{4}$ is a prime number, then $T_p$ is odd.
\end{cor}

\begin{proof}
Using the above notation, we claim that $Q_n\not\equiv 2\pmod{4}$ for any $n\in\N$. Suppose for the contrary that $Q_n\equiv 2\pmod{4}$ for some $n\in\N$. From \eqref{2} we have 
$$Q_nQ_{n+1}=p-P_{n+1}^2.$$
Since both sides of the above equation are even, we conclude that $P_{n+1}$ is odd. Since $p\equiv P_{n+1}^2\equiv 1\pmod{4}$, we obtain that both sides of the equation are divisible by $4$. Hence $Q_{n+1}$ is even as $4\nmid Q_n$. Rewriting \eqref{2} in the following way
$$Q_{k+1}=\frac{p-P_k^2}{Q_k}+2a_kP_k-a_k^2Q_k=Q_{k-1}+2a_kP_k-a_k^2Q_k, \quad k\in\N_+,$$
as $Q_{k-1}Q_k=p-P_k^2$, we see that $Q_k$ is even for each $k\geq n$. This contradicts the periodicity of the sequence $(Q_k)_{k\in\N_+}$ and the fact that $Q_l=1$.

If we suppose that $T_p=2L$ for some $L\in\N_+$, then $Q_L=2$ by Proposition \ref{T1}. However we have just proved that $Q_k\not\equiv 2\pmod{4}$ for any $k\in\N$. Thus $T_p$ is odd.
\end{proof}

Corollary \ref{C1} can also be quite easily proved from the following interesting characterization from \cite[Theorem 3]{RipTay}.

\begin{thm}[{\cite[Theorem 3]{RipTay}}]
Let $D\in\N_+$ not a perfect square. Then $T_D$ is even if and only if $D=rs$, where $r,s\in\N_+$ satisfy one of the following conditions:
\begin{enumerate}
\item $x^2r-y^2s=\pm 2$ for some odd $x,y\in\Z$;
\item $r,s\neq 1$ and $x^2r-y^2s=\pm 1$ for some $x,y\in\Z$.
\end{enumerate}
\end{thm}

In our case $D=p\equiv 1\pmod {4}$ is a prime number, so condition (2) from the above theorem is obviously not satisfied. Meanwhile, condition (1) is equivalent to the existence of a solution of the equation $x^2-py^2=\pm 2$ in odd integers. However, since $p\equiv 1\pmod {4}$, for all odd integers $x,y$ the value of $x^2-py^2$ is divisible by $4$. As a consequence, we again obtain that $T_p$ is odd.

Theorem \ref{T0} and Corollary \ref{C1} finally imply the following equivalence.

\begin{cor}\label{C2}
Let $p$ a prime number. Then
\begin{enumerate}
\item $T_p$ is even if and only if $p\equiv 3\pmod{4}$;
\item $T_p$ is divisible by $4$ if and only if $p\equiv 7\pmod{8}$.
\end{enumerate}
\end{cor}

Note that in the case $D=2p$, where $p$ is an odd prime number, the condition $p\equiv 3\pmod{4}$ implies $2\mid T_D$ (by Theorem \ref{T0}) but the inverse implication does not hold as $\sqrt{34}=[5,\overline{1,4,1,10}]$.

\bigskip

Let us further recall the polynomials $q_n\in\Z[x_1,\ldots,x_n]$ that are defined recursively  as follows:
\begin{align*}
q_{-1} & = 0,\\
q_0 & = 1,\\
q_n(x_1,\ldots,x_n) & = x_nq_{n-1}(x_1,\ldots,x_{n-1})+q_{n-2}(x_1,\ldots,x_{n-2}),\quad n\geq 1.
\end{align*}

These polynomials are connected with continued fractions via the indentity
$$[x_0,x_1,\ldots ,x_n]=\frac{q_{n+1}(x_0,x_1,\ldots,x_n)}{q_n(x_1,\ldots,x_n)}.$$

Also note that the defining recurrence can be rewritten as matrix multiplication:
\begin{equation}\label{3}
\left[\begin{array}{cc}
q_n(x_1,\ldots,x_n) & q_{n-1}(x_1,\ldots,x_{n-1})\\
q_{n-1}(x_2,\ldots,x_n) & q_{n-2}(x_2,\ldots,x_{n-1})
\end{array}\right]=\left[\begin{array}{cc}
x_1 & 1\\
1 & 0
\end{array}\right]\cdots\left[\begin{array}{cc}
x_n & 1\\
1 & 0
\end{array}\right]
\end{equation}

As an easy consequence of \eqref{3} we get
\begin{equation}\label{4}
q_n(x_1,\ldots,x_n)=q_n(x_n,\ldots,x_1)
\end{equation}
and
\begin{equation}\label{8}
q_n(x_1,\ldots,x_n)q_{n-2}(x_2,\ldots,x_{n-1})-q_{n-1}(x_1,\ldots,x_{n-1})q_{n-1}(x_2,\ldots,x_n)=(-1)^n.
\end{equation}

\section{Proof of Theorem \ref{TA}}\label{sec:3}

Let $D\in\N_+$. Write $\sqrt{D}=[a_0,\overline{a_1,\ldots,a_l}]$, where $T_D=l$ and $a_j=a_{l-j}$ for $j\in\{1,\ldots,l-1\}$. 

\

Assume that $D=p$ is a prime number. The value $a_l=2a_0$ is even, so if $a$ is odd, then it appears an odd number of times in the period of continued fraction of $\sqrt{p}$ exactly when $l=2L$ for some $L\in\N_+$ and $a_L=a$. According to Corollary \ref{C2} the number $p$ is congruent to $3$ modulo $4$. Then Theorem \ref{T0} implies that $\lfloor\sqrt{p}\rfloor\in\{a,a+1\}$, so $a<\sqrt{p}<a+2$. This means that $a^2<p<(a+2)^2$.

Conversely, let $p\equiv 3\pmod{4}$ a prime number and $a$ an odd positive integer such that $a^2<p<(a+2)^2$. Then $a<\sqrt{p}<a+2$, so $\lfloor\sqrt{p}\rfloor\in\{a,a+1\}$. Since $p\equiv 3\pmod{4}$, we have $T_p=2L$ for some $L\in\N_+$ in virtue of Corollary \ref{C2}. Theorem \ref{T0} states that $a_L\equiv p\equiv 1\pmod{2}$ and $\lfloor\sqrt{p}\rfloor\in\{a,a+1\}$. Hence $a_L=a$. Because of parity of $a_{2L}$ and palidromicity of the period of continued fraction of $\sqrt{p}$ we conclude that $a$ appears an odd number of times in this period.

Now we consider the case of even $a$. If $l=2L$ for some $L\in\N_+$, then $p\equiv 3\pmod{4}$ by Corollary \ref{C2} and $a_L$ is odd by Theorem \ref{T0}. Thus, $a$ appears an odd number of times in the period of continued fraction of $\sqrt{p}$ if and only if $a_l=a$. On the other hand, we know that $a_l=2\lfloor\sqrt{p}\rfloor$. Hence, $\lfloor\sqrt{p}\rfloor=\frac{a}{2}$, equivalently $\frac{a}{2}<\sqrt{p}<\frac{a+2}{2}$. Finally we get $\frac{a^2}{4}<p<\frac{(a+2)^2}{4}$.

Conversely, let $p$ a prime number and $a$ an even positive integer such that $\frac{a^2}{4}<p<\frac{(a+2)^2}{4}$. Put $T_p=l$. Then $\frac{a}{2}<\sqrt{p}<\frac{a+2}{2}$, so $\lfloor\sqrt{p}\rfloor =\frac{a}{2}$. Thus $a=2\lfloor\sqrt{p}\rfloor =a_l$. If $T_p=2L$ for some $L\in\N_+$, then $a_L\equiv p\equiv 1\pmod{2}$ by Theorem \ref{T0} and Corollary \ref{C2}. Because of palidromicity of the period of continued fraction of $\sqrt{p}$ we conclude that $a$ appears an odd number of times in this period.

\ 

We are left with proving the theorem for $D=2p$, where $p$ is a prime number. Since the values of $a_l=2\lfloor\sqrt{2p}\rfloor$ and $a_{l/2}$ (if $2\mid l$) are even by Theorem \ref{T0}, any odd positive integer appears an even number of times in the periodic part of continued fraction of $\sqrt{2p}$. If $a$ is an even number appearing an odd number of times in the periodic part of continued fraction of $\sqrt{2p}$, then $a\in\{a_{l/2},a_l\}$. If $a=a_{l/2}$, then $\lfloor\sqrt{2p}\rfloor\in\{a,a+1\}$, so $a<\sqrt{2p}<a+2$. This means that $a^2<2p<(a+2)^2$, equivalently $\frac{a^2}{2}<p<\frac{(a+2)^2}{2}$. If $a=a_l$, then $a=2\lfloor\sqrt{2p}\rfloor$. Hence, $\lfloor\sqrt{2p}\rfloor=\frac{a}{2}$, equivalently $\frac{a}{2}<\sqrt{2p}<\frac{a+2}{2}$. Thus we get $\frac{a^2}{4}<2p<\frac{(a+2)^2}{4}$ and finally $\frac{a^2}{8}<p<\frac{(a+2)^2}{8}$.

Now, let $p$ a prime number and $a$ an even positive integer such that $\frac{a^2}{8}<p<\frac{(a+2)^2}{8}$. Write $T_{2p}=l$. Then $\frac{a^2}{4}<2p<\frac{(a+2)^2}{4}$ and $\lfloor\sqrt{2p}\rfloor=\frac{a}{2}$. Consequently $a=2\lfloor\sqrt{2p}\rfloor =a_l$. If $2\mid l$, then $a_{\frac{l}{2}}\in\{\lfloor\sqrt{2p}\rfloor ,\lfloor\sqrt{2p}\rfloor -1\}$ by Theorem \ref{T2}. Hence, $a_{\frac{l}{2}}<a_l=a$. This fact combined with palindromicity of the period of continued fraction of $\sqrt{2p}$ implies that $a$ appears an odd number of times in this period.

Finally, let $p$ a prime number and $a$ an even positive integer such that $\frac{a^2}{2}<p<\frac{(a+2)^2}{2}$ and $T_{2p}=2L$ for some $L\in\N_+$. Then $a^2<2p<(a+2)^2$ and $\lfloor\sqrt{2p}\rfloor\in\{a,a+1\}$. By Theorem \ref{T2} we get that $a=a_L$. Since $a_{2L}=2\lfloor\sqrt{2p}\rfloor >a$, by palindromicity  of the period of continued fraction of $\sqrt{2p}$, we conclude that $a$ appears an odd number of times in this period. 
\hfill $\square$

\section{Proof of Theorem \ref{TB}}\label{sec:4}

Write $\sqrt{D}=[a_0,\overline{a_1,\ldots,a_l}]$, where $T_D=l$ is divisible by 4 and $L=\frac{l}{2}$. We split the proof into two cases depending on the parity of $\lfloor\sqrt{D}\rfloor$. We start with the simpler case when $\lfloor\sqrt{D}\rfloor$ has different parity than $D$.

\begin{proof}[Proof of Theorem $\ref{TB}$ in the case of different parity]
Assume that $\lfloor\sqrt{D}\rfloor$ has different parity than $D$.
By Theorem \ref{T2} we have $Q_L=2$ and $P_L=a_L=\lfloor\sqrt{D}\rfloor -1$. Thus 
$$\omega_L-a_L=\frac{\sqrt{D}-\lfloor\sqrt{D}\rfloor +1}{2}>\frac{1}{2}.$$
As $\omega_{L+1}=\frac{1}{\omega_L-a_L}<2$, we have $a_{L+1}=\lfloor\omega_{L+1}\rfloor =1$, finishing the proof.
\end{proof}

Note that because of palindromicity of the sequence $(a_1,\ldots,a_{l-1})$ we also conclude that $a_{L-1}=1$. Actually, the fact that $a_{L-1}=a_{L+1}=1$ if $\lfloor\sqrt{D}\rfloor\not\equiv D\pmod{2}$ also follows directly from \cite[Satz 16, page 96]{Pe}. In spite of this, we decided to demonstrate its proof for the completeness of the proof of Theorem \ref{TB}.

\bigskip

The proof of the theorem for $\lfloor\sqrt{D}\rfloor$ of the same parity as $D$ is much longer and requires the use of polynomials $q_n$. We know from Theorem \ref{T2} that the continued fraction of $\sqrt{D}$ is of the form $[a,\overline{a_1,\ldots,a_{L-1},a,a_{L-1},\ldots,a_1,2a}]$, where $a,a_1,\ldots,a_{L-1}\in\N_+$. Here $a=\lfloor\sqrt{D}\rfloor$. If $a=1$, then $1$ appears in the period of continued fraction of $\sqrt{D}$. If $a=2$, then $D\in [4,9)$ is even. It means that $D=2p$ for some prime $p$. Consequently, $D\in\{4,6\}$. If $D=4$, then $\sqrt{D}=2$ is not irrational. If $D=6$, then $\sqrt{6}=[2,\overline{2,4}]$, so $4\nmid T_6$. Thus, we can assume that $a\geq 3$. We want to expand the expression $[a,\overline{a_1,\ldots,a_{L-1},a,a_{L-1},\ldots,a_1,2a}]^2 ( = D )$ in terms of polynomials $q_n$. 

From now on till the end of this section we assume that $L,a,a_1,\ldots,a_{L-1}\in\N_+$ and $a\geq 3$.

\begin{lemma}\label{L1}
We have
\begin{align*}
&[a,\overline{a_1,\ldots,a_{L-1},a,a_{L-1},\ldots,a_1,2a}]^2\\
=&\ a^2+2a\frac{q_{L-2}(a_2,\ldots,a_{L-1})}{q_{L-1}(a_1,\ldots,a_{L-1})}+\frac{q_{L-2}(a_2,\ldots,a_{L-1})^2+2\cdot (-1)^{L-1}}{q_{L-1}(a_1,\ldots,a_{L-1})^2}\\
&+\frac{2\cdot (-1)^{L-1}(q_{L-2}(a_2,\ldots,a_{L-1})-2q_{L-2}(a_1,\ldots,a_{L-2}))}{q_{L-1}(a_1,\ldots,a_{L-1})^2(aq_{L-1}(a_1,\ldots,a_{L-1})+2q_{L-2}(a_1,\ldots,a_{L-2}))}.
\end{align*}
\end{lemma}

\begin{proof}
By \cite[Theorem 3.3]{MiUl} we know that
\begin{equation}\label{5}
\begin{split}
&[a,\overline{a_1,\ldots,a_{L-1},a,a_{L-1},\ldots,a_1,2a}]^2\\
=&\ \left(a+\frac{q_{2L-2}(a_1,\ldots,a_{L-1},a,a_{L-1},\ldots,a_2)}{q_{2L-1}(a_1,\ldots,a_{L-1},a,a_{L-1},\ldots,a_1)}\right)^2-\frac{1}{q_{2L-1}(a_1,\ldots,a_{L-1},a,a_{L-1},\ldots,a_1)^2}.
\end{split}
\end{equation}
Now we ``extract" $a$ from the polynomials $q_{2L-2}(a_1,\ldots,a_{L-1},a,a_{L-1},\ldots,a_2)$ and\linebreak $q_{2L-1}(a_1,\ldots,a_{L-1},a,a_{L-1},\ldots,a_1)$, i.e. write them as expressions of $a$ and polynomials $q_n$ depending on $a_1,\ldots,a_{L-1}$ only. The identity \eqref{3} gives us the following:
\begin{align*}
&\left[\begin{array}{cc}
q_{2L-1}(a_1,\ldots,a_{L-1},a,a_{L-1},\ldots,a_1) & q_{2L-2}(a_1,\ldots,a_{L-1},a,a_{L-1},\ldots,a_2)\\
q_{2L-2}(a_2,\ldots,a_{L-1},a,a_{L-1},\ldots,a_1) & q_{2L-3}(a_2,\ldots,a_{L-1},a,a_{L-1},\ldots,a_2)
\end{array}\right]\\
=&\ \left[\begin{array}{cc}
a_1 & 1\\
1 & 0
\end{array}\right]\cdots\left[\begin{array}{cc}
a_{L-1} & 1\\
1 & 0
\end{array}\right]\cdot\left[\begin{array}{cc}
a & 1\\
1 & 0
\end{array}\right]\cdot\left[\begin{array}{cc}
a_{L-1} & 1\\
1 & 0
\end{array}\right]\cdots\left[\begin{array}{cc}
a_1 & 1\\
1 & 0
\end{array}\right]\\
=&\ \left[\begin{array}{cc}
q_{L-1}(a_1,\ldots,a_{L-1}) & q_{L-2}(a_1,\ldots,a_{L-2})\\
q_{L-2}(a_2,\ldots,a_{L-1}) & q_{L-3}(a_2,\ldots,a_{L-2})
\end{array}\right]\cdot\left[\begin{array}{cc}
a & 1\\
1 & 0
\end{array}\right]\\
& \cdot\left[\begin{array}{cc}
q_{L-1}(a_1,\ldots,a_{L-1}) & q_{L-2}(a_2,\ldots,a_{L-1})\\
q_{L-2}(a_1,\ldots,a_{L-2}) & q_{L-3}(a_2,\ldots,a_{L-2})
\end{array}\right]\\
=&\ \left[\begin{array}{cc}
aq_{L-1}(a_1,\ldots,a_{L-1})+q_{L-2}(a_1,\ldots,a_{L-2}) & q_{L-1}(a_1,\ldots,a_{L-1})\\
aq_{L-2}(a_2,\ldots,a_{L-1})+q_{L-3}(a_2,\ldots,a_{L-2}) & q_{L-2}(a_2,\ldots,a_{L-1})
\end{array}\right]\\
&\cdot\left[\begin{array}{cc}
q_{L-1}(a_1,\ldots,a_{L-1}) & q_{L-2}(a_2,\ldots,a_{L-1})\\
q_{L-2}(a_1,\ldots,a_{L-2}) & q_{L-3}(a_2,\ldots,a_{L-2})
\end{array}\right],
\end{align*}
where in the second equality we used \eqref{4}. Hence,
\begin{equation}\label{6}
\begin{split}
&q_{2L-1}(a_1,\ldots,a_{L-1},a,a_{L-1},\ldots,a_1)\\
=&\ aq_{L-1}(a_1,\ldots,a_{L-1})^2+2q_{L-1}(a_1,\ldots,a_{L-1})q_{L-2}(a_1,\ldots,a_{L-2})
\end{split}
\end{equation}
and
\begin{equation}\label{7}
\begin{split}
&q_{2L-2}(a_1,\ldots,a_{L-1},a,a_{L-1},\ldots,a_2)\\
=&\ aq_{L-1}(a_1,\ldots,a_{L-1})q_{L-2}(a_2,\ldots,a_{L-1})+q_{L-2}(a_1,\ldots,a_{L-2})q_{L-2}(a_2,\ldots,a_{L-1})\\
&+q_{L-1}(a_1,\ldots,a_{L-1})q_{L-3}(a_2,\ldots,a_{L-2}).
\end{split}
\end{equation}
For the convenience of notation, let us write
\begin{align*}
R=q_{L-1}(a_1,\ldots,a_{L-1}),\quad & S=q_{L-2}(a_1,\ldots,a_{L-2})\\
T=q_{L-2}(a_2,\ldots,a_{L-1}),\quad & U=q_{L-3}(a_2,\ldots,a_{L-2}).
\end{align*}
Plugging \eqref{6} and \eqref{7} into \eqref{5}, we get
\begin{align*}
&[a,\overline{a_1,\ldots,a_{L-1},a,a_{L-1},\ldots,a_1,2a}]^2\\
&=\left(a+\frac{aRT+ST+RU}{aR^2+2RS}\right)^2-\frac{1}{(aR^2+2RS)^2}\\
&=\left(a+\frac{T}{R}+\frac{RU-ST}{aR^2+2RS}\right)^2-\frac{1}{(aR^2+2RS)^2}\\
&=\left(a+\frac{T}{R}+\frac{(-1)^{L-1}}{aR^2+2RS}\right)^2-\frac{1}{(aR^2+2RS)^2}\\
&=a^2+2a\frac{T}{R}+\frac{2\cdot (-1)^{L-1}a}{aR^2+2RS}+\frac{2\cdot (-1)^{L-1}T}{R(aR^2+2RS)}+\frac{T^2}{R^2}\\
&=a^2+2a\frac{T}{R}+\frac{T^2+2\cdot (-1)^{L-1}}{R^2}+\frac{2\cdot (-1)^{L-1}(T-2S)}{R^2(aR+2S)},
\end{align*}
where we used \eqref{8} in the third equality. This was to prove.
\end{proof}

Remember that $[a,\overline{a_1,\ldots,a_{L-1},a,a_{L-1},\ldots,a_1,2a}]^2=D\in\Z$. Thus we need to have a necessary condition for the expression $[a,\overline{a_1,\ldots,a_{L-1},a,a_{L-1},\ldots,a_1,2a}]^2$ to be integral.

\begin{cor}\label{C3}
If $$[a,\overline{a_1,\ldots,a_{L-1},a,a_{L-1},\ldots,a_1,2a}]^2\in\Z,$$ then $$q_{L-2}(a_2,\ldots,a_{L-1})=2q_{L-2}(a_1,\ldots,a_{L-2}).$$
\end{cor}

\begin{proof}
Use the notation from the previous proof. By Lemma \ref{L1}, we have
\begin{align*}
&[a,\overline{a_1,\ldots,a_{L-1},a,a_{L-1},\ldots,a_1,2a}]^2\\
=&\ a^2+2a\frac{T}{R}+\frac{T^2+2\cdot (-1)^{L-1}}{R^2}+\frac{2\cdot (-1)^{L-1}(T-2S)}{R^2(aR+2S)}.
\end{align*}
If the above value is integral, then all the summands in the right hand side of the above equality can be written as quotients of integers with common denominator $R^2$. This means that $aR+2S$ divides $2(T-2S)$. However,
$$aR+2S>3T+2S\geq 2T-4S$$
and
$$aR+2S>3S+2S\geq 4S-2T.$$
In other words, $aR+2S>|2(T-2S)|$. This fact combined with the divisibility of $2(T-2S)$ by $aR+2S$ implies that $T=2S$.
\end{proof}

Corollary \ref{C3} has been already stated and proved as Satz 28 in \cite[page 112]{Pe}. However, we decided to show its proof via Lemma \ref{L1} for the completeness of the proof of Theorem \ref{TB}.

Now we recall that $L=2m$, $m\in\N_+$ as $4\mid T_D=2L$. With this additional assumption we show a necessary condition for the equality $q_{L-2}(a_2,\ldots,a_{L-1})=2q_{L-2}(a_1,\ldots,a_{L-2})$.

\begin{lemma}\label{L2}
If $q_{2m-2}(a_2,\ldots,a_{2m-1})=2q_{2m-2}(a_1,\ldots,a_{2m-2})$, then $a_j=1$ for some $j\in\{1,\ldots,2m-1\}$.
\end{lemma}

\begin{proof}
We prove the lemma by induction on $m\in\N_+$. For $m=1$ the assumption of the implication in the statement of lemma
$1=q_0=2q_0=2$ is not satisfied, so the implication is true. Let $m>1$ and assume that $q_{2m-2}(a_2,\ldots,a_{2m-1})=2q_{2m-2}(a_1,\ldots,a_{2m-2})$. Perform the sequence of equivalent transforms of the equality being our assumption.
\begin{equation}\label{9}
\begin{split}
&q_{2m-2}(a_2,\ldots,a_{2m-1})=2q_{2m-2}(a_1,\ldots,a_{2m-2})\\
\Leftrightarrow &\ a_{2m-1}q_{2m-3}(a_2,\ldots,a_{2m-2})+q_{2m-4}(a_2,\ldots,a_{2m-3})\\
&=2a_1q_{2m-3}(a_2,\ldots,a_{2m-2})+2q_{2m-4}(a_3,\ldots,a_{2m-2})\\
\Leftrightarrow &\ q_{2m-4}(a_2,\ldots,a_{2m-3})-2q_{2m-4}(a_3,\ldots,a_{2m-2})=(2a_1-a_{2m-1})q_{2m-3}(a_2,\ldots,a_{2m-2})
\end{split}
\end{equation}

If $a_2=1$, then we are done. Hence assume that $a_2\geq 2$. If both sides of the last equality are nonzero, then the right hand side has absolute value at least equal to $q_{2m-3}(a_2,\ldots,a_{2m-2})$, which is greater than or equal to $\max\{q_{2m-4}(a_2,\ldots,a_{2m-3}),2q_{2m-4}(a_3,\ldots,a_{2m-2})\}$. Since both values $q_{2m-4}(a_2,\ldots,a_{2m-3}),2q_{2m-4}(a_3,\ldots,a_{2m-2})$ are positive, we conclude that
$$q_{2m-3}(a_2,\ldots,a_{2m-2})>|q_{2m-4}(a_2,\ldots,a_{2m-3})-2q_{2m-4}(a_3,\ldots,a_{2m-2})|,$$
so the last equality in \eqref{9} (and thus all the equalities in \eqref{9}) does not hold. Consequently, both sides of the last equality in \eqref{9} are zero. By inductive assumption for $m-1$, since $q_{2m-4}(a_2,\ldots,a_{2m-3})=2q_{2m-4}(a_3,\ldots,a_{2m-2})$, there exists $j\in\{2,\ldots,2m-3\}$ such that $a_j=1$.
\end{proof}

Now we obtain the proof of Theorem \ref{TB} in the case of the same parity of $\lfloor\sqrt{D}\rfloor$ and $D$ by combining the preceding results.

\begin{proof}[Proof of Theorem $\ref{TB}$ in the case of the same parity]
Write $$\sqrt{D}=[\lfloor\sqrt{D}\rfloor,\overline{a_1,\ldots,a_{2m-1},\lfloor\sqrt{D}\rfloor,a_{2m-1},\ldots,a_1,2\lfloor\sqrt{D}\rfloor}].$$ Since $D=[\lfloor\sqrt{D}\rfloor,\overline{a_1,\ldots,a_{2m-1},\lfloor\sqrt{D}\rfloor,a_{2m-1},\ldots,a_1,2\lfloor\sqrt{D}\rfloor}]^2$ is an integer, Corollary \ref{C3} gives us equality $q_{2m-2}(a_2,\ldots,a_{2m-1})=2q_{2m-2}(a_1,\ldots,a_{2m-2})$. By Lemma \ref{L2} this equality holds only if $a_j=1$ for some $j\in\{1,\ldots ,2m-1\}$.
\end{proof}

A more careful analysis of the above reasoning allows us to state that if $q_{2m-2}(a_2,\ldots,a_{2m-1})=2q_{2m-2}(a_1,\ldots,a_{2m-2})$, then $a_j=1$ for some $$j\in\left\{2,4,\ldots,2\left\lceil\frac{m}{2}\right\rceil-2,m,2\left\lfloor\frac{m}{2}\right\rfloor+3,\ldots,2m-5,2m-3\right\}.$$ As a result, if $p$ is a prime number and $D\in\{p,2p\}$ is such that $D\equiv\lfloor\sqrt{D}\rfloor =a\pmod{2}$, $T_D=4m$ for some $m\in\N_+$ and $\sqrt{D}=[a,\overline{a_1,\ldots,a_{2m-1},a,a_{2m-1},\ldots,a_1,2a}]$, then $a_j=1$ for some $$j\in\left\{2,4,\ldots,2\left\lceil\frac{m}{2}\right\rceil-2,m,2\left\lfloor\frac{m}{2}\right\rfloor+3,\ldots,2m-5,2m-3\right\}.$$

In view of the proof of Theorem \ref{TB} in the case of $\lfloor\sqrt{D}\rfloor\equiv D\pmod{2}$ and the discussion above we can make the statement of Theorem \ref{TB}  more specific:

\begin{thm}\label{TA2}
Let $p$ a prime number, $D\in\{p,2p\}$, and $T_D=4m$ for some $m\in\N_+$. Then $1$ appears in the period of continued fraction of $\sqrt{D}$. In particular, if $$\sqrt{D}=[a,\overline{a_1,\ldots,a_{2m-1},a,a_{2m-1},\ldots,a_1,2a}],$$ then
\begin{enumerate}
\item $a_{2m-1}=a_{2m+1}=1$ when $a\not\equiv D\pmod{2}$; 
\item $a_j=1$ for some $j\in\left\{2,4,\ldots,2\left\lceil\frac{m}{2}\right\rceil-2,m,2\left\lfloor\frac{m}{2}\right\rfloor+3,\ldots,2m-5,2m-3\right\}$ when $a\equiv D\pmod{2}$.
\end{enumerate}
\end{thm}

Note that one can also reason analogously in the cases of prime powers and to show the following.

\begin{thm}
Let $p$ a prime number. Let $D$ an integer of the form $p^n$ or $2p^n$, where $n$ is an odd positive integer. Assume that $2\mid T_D$. Write $T_D=2L$ for some $L\in\N_+$ and $\sqrt{D}=[\lfloor\sqrt{D}\rfloor,\overline{a_1,\ldots,a_{L-1},a_L,a_{L-1},\ldots,a_1,2\lfloor\sqrt{D}\rfloor}]$. If $\lfloor\sqrt{D}\rfloor\not\equiv D\pmod{2}$, then $a_{L-1}=1$.
\end{thm}

\begin{proof}
By the proof of \cite[Satz 21, page 108]{Pe}, we again have that $Q_L=2$. Then, Proposition \ref{P2} gives us that $P_L=a_L\in\{\lfloor\sqrt{D}\rfloor,\lfloor\sqrt{D}\rfloor-1\}$. Since $\lfloor\sqrt{D}\rfloor\not\equiv D\pmod{2}$, we get from Proposition \ref{P1} that $P_L=a_L=\lfloor\sqrt{D}\rfloor-1$. Then 
$$\omega_L-a_L=\frac{\sqrt{D}-\lfloor\sqrt{D}\rfloor +1}{2}>\frac{1}{2}.$$
As $\omega_{L+1}=\frac{1}{\omega_L-a_L}<2$, we have $a_{L+1}=\lfloor\omega_{L+1}\rfloor =1$.
\end{proof}

The above result together with Theorem \ref{T0} shows that if $1$ does not appear as a partial quotient in continued fraction expansion of square root of prime number $p\equiv 3\pmod{8}$, then $\lfloor\sqrt{p}\rfloor$ is odd and the middle partial quotient in the period of continued fraction of $\sqrt{p}$ is equal to $\lfloor\sqrt{p}\rfloor$. This fact is confirmed by the numerical data contained in \cite[Table 4]{MiUl}.

Let us recall that if $p$ is a prime number congruent to $3$ modulo $8$, then $T_p=2L$ for some odd $L\in\N_+$. In the case $p\equiv 3\pmod{8}$ the reasoning in the proof of Theorem \ref{TB} is valid up to Corollary \ref{C3} included. Only the statement of Lemma \ref{L2} may be false. It is possible in the case of odd $L$ that $q_{L-2}(a_2,\ldots,a_{L-1})=2q_{L-2}(a_1,\ldots,a_{L-2})$ but $a_j>1$ for all $j\in\left\{1,\ldots,L-1\right\}$. Analogous reasoning to this one from the proof of Lemma \ref{L2} shows that $q_{L-2}(a_2,\ldots,a_{L-1})=2q_{L-2}(a_1,\ldots,a_{L-2})$ if $\frac{3-(-1)^k}{2}a_k=\frac{3+(-1)^k}{2}a_{L-k}$ for each $k\in\left\{1,\ldots,\frac{L-1}{2}\right\}$. In other words, $q_{L-2}(a_2,\ldots,a_{L-1})=2q_{L-2}(a_1,\ldots,a_{L-2})$ if $2a_k=a_{L-k}$ for each odd $k\in\left\{1,\ldots,\frac{L-1}{2}\right\}$ and $a_k=2a_{L-k}$ for each even $k\in\left\{1,\ldots,\frac{L-1}{2}\right\}$. This phenomenon can be also observed in \cite[Table 4]{MiUl}.

\section{Proof of Theorem \ref{th:period length}}\label{sec:5}

In this section we estimate the period length, giving a partial answer to \cite[Question 5.1]{MiUl}.

\begin{proof}[Proof of Theorem $\ref{th:period length}$]
We will use the class number formula 
$$h=\frac{\sqrt D}{\log \varepsilon}L(1,\chi),$$
where $h$ and $\varepsilon>1$ are the class number and fundamental unit of the real quadratic field $\Q(\sqrt D)$, respectively.
For the class number, we will just use the trivial estimate $h\geq 1$.

For the fundamental unit, first let $p_i/q_i=[a_0,a_1,\dots,a_i]$ denote the convergents to $\sqrt D$ and define $\alpha_i=p_{i}+q_{i}\sqrt D$. A well-known fact then says that $\alpha_{l-1}$ is a unit in $\Q(\sqrt D)$ that equals $\ve$ or $\ve^2$. Also note that, in the notation of Section \ref{sec:2}, $p_i=q_{i+1}(a_0,a_1,\dots,a_i)$ and $q_i=q_{i}(a_1,\dots,a_i)$.

To estimate $\alpha_{l-1}$, we use the recurrences $\alpha_{i+1}=a_{i+1}\alpha_i+\alpha_{i-1}$ that can be extended to hold with the initial conditions $\alpha_0=a_0+\sqrt D, \alpha_{-1}=1$. 

Similarly as, e.g., in \cite[Lemma 8]{DK}, using the recurrence twice, we get 
$$\alpha_{i+1}=a_{i+1}\alpha_i+\alpha_{i-1}=(a_{i+1}a_i+1)\alpha_{i-1}+a_{i+1}\alpha_{i-2}
> (a_{i+1}a_i+1)\alpha_{i-1}\geq 2\alpha_{i-1}$$
(it is easy to check that the estimate holds also for $i=0$).

Thus $\alpha_{l-1}>2\alpha_{l-3}>\dots>
\begin{cases}
  2^{(l-1)/2}\alpha_0 &  \text{ if } l \text{ is odd,} \\
  2^{l/2}\alpha_{-1}  &  \text{ if } l \text{ is even.}
\end{cases}
$

In both cases, we get $\alpha_{l-1}>2^{l/2}$.

Thus 
$\log\ve\geq \frac12\log\alpha_{l-1}>l\frac{\log2}{4}$, and so $$l<\frac 4{\log 2}\log\ve\leq \frac 4{\log 2}\sqrt D \cdot L(1,\chi).$$

Finally, for the $L$-value, we have the well-known bounds $L(1,\chi)\ll\log D$ (unconditionally) and $L(1,\chi)\ll \log\log D$ (under GRH), see, e.g., \cite{GS}.
\end{proof}

As we noted in the Introduction, similar results are known, e.g., \cite{Pod,Poh,Sa}. Nevertheless, we included the proof here for completeness of the answer to Question 5.1 in \cite{MiUl}.

This question asked for estimates of the period length $T_{p_m}$, where $p_m$ is the $m$th prime. As $p_m \sim m\log m$ by the prime number theorem, we immediately get from Theorem \ref{th:period length} above that 
 $$T_{p_m}\ll
\begin{cases}
  m^{1/2}(\log m)^{3/2}  &  \text{ unconditionally}, \\
  m^{1/2}(\log m)^{1/2}\log\log m  &  \text{ assuming GRH}.
\end{cases}$$

Thus it seems that unconditionally answering part 1. of the question (that asks whether always $T_{p_m}<m^{1/2}\log m$) will be very hard, as it corresponds to establishing an unconditional upper bound $L(1,\chi)\ll\sqrt{\log p}$.

Assuming GRH, one can use an explicit form of Littlewood's bound 
\cite[Theorem 1.5]{LLS}. As $\chi$ is a primitive character modulo $D$ or $4D$, plugging in this bound to our Theorem \ref{th:period length} and simplifying, we get
$$T_D<21\sqrt D(\log\log (4D)+0.2)$$
for $D>10^{10}$.
If we let $D=p_m$ the $m$th prime, we can use the explicit upper bound $p_m<m(\log m+\log\log m)$ valid for $m\geq 6$ \cite[Corollary to Theorem 3]{RS}. Comparing the resulting bound on $T_{p_m}$ with \cite[Question 5.1, part 1.]{MiUl}, we see that if $m>10^{24000}$, then the conjectured bound indeed holds. Unfortunately, this bound is so large that computationally checking what happens for all smaller values of $m$ is most likely impossible.

Of course, one could try to improve the other estimates used in the proof of Theorem \ref{th:period length}. But $\alpha_{l-1}>2^{l/2}$ is essentially sharp (when all the partial quotient $a_1,\dots,a_{l-1}$ are small), and in such a case it can most likely also happen that $D$ is prime (for $D=D(t)$ will be given by a quadratic polynomial that conjecturally attains prime values, cf. \cite[Proposition 6]{DK}, \cite[Theorem 3.3]{MiUl}, and Section \ref{factorpoly}), as well as that the class number is small (when the parameter $t$ is small).

\

Part 2 of \cite[Question 5.1]{MiUl} asks about the value
$$\limsup_{m\to\infty}\frac{T_{p_m}}{m^{1/2}\log m}.$$
Here again we cannot say anything unconditionally, but we see that our estimate under GRH immediately implies that the $\limsup$ equals 0.

\section{Factorization of continued fraction polynomials}\label{factorpoly}

Let us now focus on ``continued fraction families (or polynomials)'': Let $l\geq 1$ and fix a symmetric sequence of positive integers $a_1,a_2,\dots,a_{l-1}=a_1$ (when $l=1$, then the sequence is empty).

Friesen \cite{Fr} and Halter-Koch \cite{HK} determined conditions under which there exist (squarefree) integers $D$ such that $\sqrt{D}=[a_0,\overline{a_1,\ldots,a_l}]$ for some $a_0$ and $a_l=2a_0$. It turns out that these conditions depend only on certain parities of expressions related to the integers $q_i(a_1,a_2,\dots,a_i)$. 

To be precise, 
define the positive integers $E, F, G$ by the matrix equation
\begin{equation}\label{eq:EFG}
\left[ {\begin{array}{cc}
    E & F \\
    F & G \\
  \end{array} } \right]=
  \prod_{i=1}^{l-1} \left[ {\begin{array}{cc}
    a_i & 1 \\
    1 & 0 \\
  \end{array} } \right].
\end{equation} 
  As in identity \eqref{3}, we have $E=q_{l-1}(a_1,a_2,\dots,a_{l-1}), F=q_{l-2}(a_1,a_2,\dots,a_{l-2})$ and 
$G=q_{l-3}(a_2,a_3,\dots,a_{l-2})$.
 
With this notation, Friesen's theorem says:

\begin{thm}[{\cite[Theorem]{Fr}}]\label{th:friesen}
Let $l\geq 1$ and fix a symmetric sequence of positive integers $a_1,a_2,\dots,a_{l-1}=a_1$. Let $E,F,G$ as above.
There are infinitely many squarefree integers $D>1$ such that $\sqrt{D}=[a_0,\overline{a_1,\ldots,a_l}]$ for some $a_0$ and $a_l=2a_0$ if and only if $F$ or $G$ is even. Furthermore, all such $D$ are of the form $D=D(t)=at^2+bt+c$ for $t\geq t_0$, where
\begin{enumerate}
    \item $a=E^2, b=2F-(-1)^lEFG, c=(F^2/4-(-1)^l)G^2$
    if $F$ is even and $E,G$ are odd, 
    or $E,F$ are odd and $G$ is even;
    \item $a=E^2/4, b=F-(-1)^lEFG/2, c=(F^2/4-(-1)^l)G^2$ if $F$ is odd and $E, G$ are even;
\end{enumerate}
and $t_0=\lfloor(-1)^lFG/E \rfloor +1$. Then $a_0=\left(Et-FG\right)/2$.
\end{thm}

Note that $F^2-EG=(-1)^l$ thanks to \eqref{eq:EFG}. Thus although Friesen stated his condition in terms of $(F^2-(-1)^l)/E$, this equals $G$ that we use in our formulation.

Further, the equality $F^2-EG=(-1)^l$ allows us to determine the possible parities above: for example, in case (1), if $F$ is even, then $E,G$ have to be odd as stated. 

It is also worth remarking that the discriminant of $D(t)$ as a quadratic polynomial is $(-1)^l\cdot 4$ in case (1), and $(-1)^l$ in case (2).

Note that, unfortunately, the reformulation of the previous theorem given by Dahl and the first author as \cite[Proposition 6]{DK} contains an error: In the statement of the proposition, there should not be $q_{s-3}$, but $G=q_{s-3}(a_2,a_3,\dots,a_{s-2})$ as defined above (\cite{DK} denotes the period length $s$ instead of $l$ in the present paper). This misprint has no impact on the rest of the paper \cite{DK}.

\

Inspired by the factorization of \cite[Theorem 3.3]{MiUl}, we can similarly -- and more explicitly -- factor $D(t)$ when the period length $l$ is even. Also note that \cite[Lemma 3b)]{RipTay} gives an equivalent factorization as our theorem (but does not consider what happens in these families).

\begin{thm}\label{factor}
In the setting of Theorem $\ref{th:friesen}$, assume moreover that $l$ is even. Let cases $(1)$ and $(2)$ below as in Theorem $\ref{th:friesen}$. Further let
\begin{enumerate}
\item $E_{\pm}=(F\pm 1, E)$, $G_{\pm}=(F\pm 1, G)$, and
\item $2E_{\pm}=(F\pm 1, E)$, $2G_{\pm}=(F\pm 1, G)$.
\end{enumerate}

Then we have the following factorizations into products of integers:
\begin{enumerate}
\item $D=D(t)=\left(E_+^2t-G_-^2(F+2)/2\right)\left(E_-^2t-G_+^2(F-2)/2\right)$,
\item $D=D(t)=\left(E_+^2t-G_-^2(F+2)\right)\left(E_-^2t-G_+^2(F-2)\right)$.
\end{enumerate}

When $t\geq t_0$, then both parentheses in the factorization are positive integers.

Moreover, if $D=D(t)$ is prime, then $t=t_0$ attains the smallest possible value.
\end{thm}

\begin{proof}
The factorizations are straightforward to verify by expanding and simplifying using $F^2-EG=1$:
In case (1), $F$ is even, and so $F-1$ and $F+1$ are coprime. Thus $(F-1)(F+1)=EG$ implies that $F\pm 1=E_\pm G_\pm$ and $E=E_-E_+, G=G_-G_+$. Using this, one just expands the factorization and checks that it equals $D(t)$ from Theorem \ref{th:friesen}.

Similarly in case (2), except that we have
$F\pm 1=2E_\pm G_\pm$ and $E=2E_-E_+, G=2G_-G_+$

Also the positivity of the parentheses is easily checked directly with the use of identities $F\pm 1=E_\pm G_\pm$ and $E=E_-E_+, G=G_-G_+$ ($F\pm 1=2E_\pm G_\pm$ and $E=2E_-E_+, G=2G_-G_+$, respectively).

If $D$ is prime, then one of the parentheses equals 1, and the other one $>1$. But this can happen only for the smallest possible value $t=t_0$, as both parentheses are strictly increasing functions of $t$.
\end{proof}

One can also directly compute that the possible prime value equals:
\begin{enumerate}
    \item $\frac{E_+^2-2}{E_-^2}$ or $\frac{E_-^2+2}{E_+^2}$;
    \item $\frac{E_+^2-1}{E_-^2}$ or $\frac{E_-^2+1}{E_+^2}$.
\end{enumerate}
Since a square of an integer is congruent to $0$, $1$ or $4$ modulo $8$, it is easy to compute that if $\frac{E_\mp^2\pm c}{E_\pm^2}$, $ c\in\{1,2\}$, is a prime number, then
\begin{itemize}
    \item $\frac{E_+^2-2}{E_-^2}\equiv 7\pmod{8}$;
    \item $\frac{E_-^2+2}{E_+^2}\equiv 3\pmod{8}$;
    \item $\frac{E_-^2+1}{E_+^2}\equiv 1\pmod{4}$.
\end{itemize}
Note that for a given integer $D$, to be of the form $\frac{x^2- c}{y^2}$, where $x,y, c\in\Z$, is equivalent to solubility of Pell-type equation $x^2-Dy^2= c$ in integers $x,y$. This corresponds with the facts \cite[S\"{a}tze 22--24]{Pe} that if $p$ is a prime number, then
\begin{itemize}
    \item $x^2-py^2=2$ has a solution in integers $x,y$ if and only if $p\equiv 7\pmod{8}$;
    \item $x^2-py^2=-2$ has a solution in integers $x,y$ if and only if $p\equiv 3\pmod{8}$;
    \item $x^2-py^2=-1$ has a solution in integers $x,y$ if and only if $p\equiv 1\pmod{4}$.
\end{itemize}

Recall that if $\sqrt{D}=[a_0,\overline{a_1,\ldots,a_l}]$, then we write $\omega_k=[\overline{a_k,\ldots,a_l,a_1,\ldots,a_{k-1}}]$ for $k\in\N_+$. We already know that for each $k\in\N_+$ we have $\omega_k=\frac{\sqrt{D}+P_k}{Q_k}$ for some $P_k,Q_k\in\N_+$ such that $P_k^2<p$ and $Q_k\mid D-P_k^2$. The reasoning in \cite[page 102]{Pe} shows that $(-1)^{k}Q_k$, $k\in\N$, is represented as a value of binary quadratic form $x^2-Dy^2$ for some integral $x$ and $y$ (e.g., by taking $x=q_{k+1}(a_0,a_1,a_2,\ldots,a_k)$ and $y=q_{k}(a_1,a_2,\ldots,a_k)$). Assuming that $D=p$ is a prime number, by Corollary \ref{C2}a) we know that $2\nmid l$ if and only if $p\equiv 1\pmod{4}$. Hence, if $p\equiv 1\pmod{4}$, putting $k=l$ we obtain that the equation $x^2-py^2=-1$ has a solution in integers $x,y$. In the case of $p\equiv 3\pmod{4}$ we have $l=2L$ for some $L\in\N_+$. As we have proved in Proposition \ref{T1}, $Q_L=2$. This fact combined with Corollary \ref{C2}b) gives us that if $p\equiv 5\pm 2\pmod{8}$, then the equation $x^2-py^2=\pm 2$ has a solution in integers $x,y$.

Since $l$ in the assumption of Theorem \ref{factor} is even and a prime number of the form $\frac{E_-^2+1}{E_+^2}$ is congruent to $1$ modulo $4$, Corollary \ref{C2} allows us to conclude that the equation $E_+^2t_0-G_-^2(F+2)=1$ does not lead to prime value of $D=D(t_0)=\left(E_+^2t_0-G_-^2(F+2)\right)\left(E_-^2t_0-G_+^2(F-2)\right)=\frac{E_-^2+1}{E_+^2}.$ Thus, if $D(t_0)$ is a prime number in the case (2) of Theorem \ref{factor}, then $E_-^2t_0-G_+^2(F-2)=1$ and $D(t_0)=\frac{E_+^2-1}{E_-^2}$.

There is also another restriction for the value of $D(t_0)$ to be prime. According to Theorems \ref{th:friesen} and \ref{factor} we get $a_0=(Et_0-FG)/2$ while Theorem \ref{T0} gives us a relation $a_L\leq a_0\leq a_L+1$. As a result, if $D(t_0)$ is a prime number, then $a_L\leq (Et_0-FG)/2\leq a_L+1$, where $E,F,G$ and $t_0$ are expressions depending on $a_1,\ldots,a_L$.

\section*{Acknowledgments}
We are grateful to the anonymous referee for the remarks that improved the exposition of the paper. We thank Maciej Ulas for several useful comments, in particular for suggesting that it is possible to include also the case $2p$ in Theorems \ref{TA}, \ref{TB}, and to Giacomo Cherubini for pointing out several missing references.


\begin{thebibliography}{WWW}

\bibitem[CS]{CS} D. Chakraborty and A. Saikia, \textit{On a conjecture of Mordell}, Rocky Mountain J. Math. \textbf{49} (2019), 2545--2556.

\bibitem[CF+]{CF+} G. Cherubini, A. Fazzari, A. Granville, V. Kala, P. Yatsyna, \textit{Consecutive real quadratic fields with large class numbers}, Int. Math. Res. Not. IMRN (to appear), \url{https://doi.org/10.1093/imrn/rnac176}.

\bibitem[DK]{DK}
A.~Dahl, V.~Kala,
\emph{Distribution of class numbers in continued fraction families of real quadratic fields},
Proc.~Edinb.~Math.~Soc.~(2) \textbf{61} (2018), no.~4, 1193--1212. 

\bibitem[DCS]{DCS}  S. Das, D. Chakraborty and A. Saikia, \textit{On the period of the continued fraction of $\sqrt{pq}$}, Acta Arith. \textbf{196} (2020), 291--302.

\bibitem[Fr]{Fr} C. Friesen, \emph{On continued fractions of given period}, Proc. Amer. Math. Soc. \textbf{103} (1988), 8--14.

\bibitem[Gol]{Gol} E. P. Golubeva, \textit{Quadratic irrationals with fixed period length in the continued fraction expansion}, J. Math. Sci. 70 (1994), 2059–-2076.

\bibitem[GS]{GS} A. Granville, K. Soundararajan, 
\textit{Upper bounds for $|L(1,\chi)|$},
Q. J. Math. \textbf{53} (2002), 265--284. 


\bibitem[HK]{HK} F. Halter-Koch, \emph{Continued fractions of given symmetric period}, Fibonacci Quart. \textbf{29} (1991), 298--303.

\bibitem[Kuz]{Kuz} R. Kuzmin, \textit{Ob odnoi zadache Gaussa}, Doklady Akad. Nauk, Ser. A. (1928), 375–-380.

\bibitem[LLS]{LLS}
Y. Lamzouri, X. Li, K. Soundararajan, \textit{Conditional bounds for the least quadratic non-residue and related problems}, Math. Comp. \textbf{84} (2015), 2391--2412.

\bibitem[Lo1]{Lou20} S. Louboutin, \textit{On the continued fraction expansions of $\sqrt{p}$ and $\sqrt{2p}$ for primes $p \equiv 3 \pmod{4}$}, Class groups of Number fields and related topics (2020), 175--178.

\bibitem[Lo2]{Lou} S. Louboutin, \textit{On the continued fraction expansions of $\frac{1+\sqrt{pq}}{2}$ and $\sqrt{pq}$}, C. R. Math., Académie des sciences (Paris), 359(9) (2021), 1201--1205.

\bibitem[MU]{MiUl} P. Miska, M. Ulas, \textit{On Consecutive 1’s in Continued Fractions Expansions of Square Roots of Prime Numbers}, Exp. Math., 31(1) (2022), 238--251.

\bibitem[Pe]{Pe} O. Perron, \emph{Die Lehre von den Kettenbr\" uchen}, Band 1, B. G. Teubner, Stuttgart, 1954.

\bibitem[Pod]{Pod} E. V. Podsypanin, \textit{Length of the period of a quadratic irrational}, J. Soviet Math. 18 (1982), 919--923.

\bibitem[Poh]{Poh} A. Pohl, \textit{An upper bound for the period length of a quadratic irrational}, Abh. Math. Sem. Univ. Hamburg 77 (2007), 129--136.

\bibitem[RS]{RS} J. B. Rosser, L. Schoenfeld, \textit{Approximate formulas for some functions of prime numbers}, Illinois J. Math. 6 (1962), 64--94.

\bibitem[RT]{RipTay} P. J. Rippon, H. Taylor, \textit{Even and odd periods in continued fractions of square roots}, Fibonacci Q., 42(2) (2004), 170--180.

\bibitem[Sa]{Sa} N. Saradha, \textit{On the length of the period of a real quadratic irrational}, Indian J. Pure Appl. Math. 48 (2017), 311--321. 

\bibitem[Ska]{Ska} M. Skałba, \textit{On the equation $a^2 + bc = n$ with
restricted unknowns}, Bull. Polish Acad. Sci. Math. 64 (2016), 137–-145.
(2016), 137--145.
\end{thebibliography}
\end{document}